\documentclass[leqno]{amsart}
\usepackage{fancyhdr}
\usepackage{amsmath}
\usepackage{amssymb}
\usepackage{amsthm}
\usepackage{bbm}
\usepackage{hyperref}
\usepackage{esint}
\usepackage[margin=1in]{geometry}
\usepackage{arydshln}

\newcommand{\bb}[1]{\mathbb{#1}}
\newcommand{\cc}[1]{\mathcal{#1}}

\newcommand{\bN}{{\mathbb N}}

\newcommand{\bR}{{\mathbb R}}

\newcommand{\abs}[1]{\left|{#1}\right|}

\newcommand{\paren}[1]{\left({#1}\right)}
\newcommand{\diver}{\operatorname{div}}
\newcommand{\vol}{\operatorname{vol}}
\newcommand{\Ric}{\operatorname{Ric}}

\newcommand{\KN}{\mathbin{\bigcirc\mspace{-15mu}\wedge\mspace{3mu}}}

\newtheorem{theorem}{Theorem}[section]

\newtheorem{lemma}[theorem]{Lemma}
\newtheorem{prop}[theorem]{Proposition}
\newtheorem{definition}[theorem]{Definition}
\newtheorem{remark}[theorem]{Remark}

\newtheorem*{namedtheorem}{\theoremname}
\newcommand{\theoremname}{testing}
\newenvironment{named}[1]{\renewcommand{\theoremname}{#1}\begin{namedtheorem}}{\end{namedtheorem}}

\begin{document}

\title{Quasiconformal Flows on non-Conformally Flat Spheres}

\author{Sun-Yung Alice Chang}

\author{Eden Prywes}

\author{Paul Yang}

\thanks{Sun-Yung Alice Chang and Paul Yang were partially supported by the NSF grant DMS-1607091.  Paul Yang was also supported by the Simons Foundation grant 615589.  Eden Prywes was partially supported by the NSF grant RTG-DMS-1502424.}

\date{\today}
\begin{abstract}
     We study integral curvature conditions for a Riemannian metric $g$ on $S^4$ that quantify the best bilipschitz constant between $(S^4,g)$ and the standard metric on $S^4$.
     Our results show that the best bilipschitz constant is controlled by the $L^2$-norm of the Weyl tensor and the $L^1$-norm of the $Q$-curvature, under the conditions that those quantities are sufficiently small, $g$ has a positive Yamabe constant and the $Q$-curvature is mean-positive.
     The proof of the result is achieved in two steps.  Firstly, we construct a quasiconformal map between two conformally related metrics in a positive Yamabe class.  Secondly, we apply the Ricci flow to establish the bilipschitz equivalence from such a conformal class to the standard conformal class on $S^4$.

\end{abstract}
\maketitle
\section{Introduction}

A bilipschitz map between two metric spaces $(X,d_X)$ and $(Y,d_Y)$ is a map $f \colon X \to Y$ for which there exists a constant $L > 0$ so that for all $x_1,x_2 \in X$,
\begin{align*}
    \frac{1}{L}d_X(x_1,x_2) \le d_Y(f(x_1),f(x_2)) \le Ld_X(x_1,x_2).
\end{align*}
Let $g$ be a smooth Riemannian metric on $S^4$.
In this paper we give integral curvature conditions which control the best bilipschitz constant between $(S^4,g)$ and $S^4$ equipped with the standard spherical metric.  We show that the dependence of $L$ can be controlled by the conformally invariant $L^2$-norm of the Weyl tensor and the $L^1$-norm of the $Q$-curvature whenever those norms are sufficiently small.
The only non-conformally invariant quantity that controls the bilipschitz constant is the $Q$-curvature of the metric.

The conformally invariant Yamabe constant, $Y([g])$ is defined as
\begin{align*}
    Y([g]) = \inf_{\hat g = e^{2w}g} \frac{1}{\operatorname{vol}_{\hat g}(S^4)^{1/2}}\int R_{\hat g} dv_{\hat g}.
\end{align*}
We prove the following theorem for metrics with non-negative Yamabe constant.
\begin{theorem}\label{thm-qcflowbilip}
Let $g_1$ be a metric so that $Y([g_1]) \ge 0$.  Let $g_0$ be a metric in the conformal class of $g_1$.
Suppose additionally that
\begin{align}
    \int_{S^4}Q_1 dv_1 \ge 0,
\end{align}
where $Q_1$ is the $Q$-curvature of $g_1$.
Let $w \colon S^4 \to \bR$ be the smooth function such that $g_1 = e^{2w}g_0$. 
There exists $\epsilon_1 > 0$ such that if
\begin{align*}
    \alpha = \int_{S^4} |Q_1e^{4w} - Q_0|dv_0 < \epsilon_1,
\end{align*}
then there exists a quasiconformal and bilipschitz map $f \colon (S^4,g_1) \to (S^4,g_0)$ whose bilipschitz constant depends only on $\alpha$ and $g_0$.  More precisely, the bilipschitz constant $L_1$ of $f$ satisfies
\begin{align*}
    L_1 \le 1 + C \alpha,
\end{align*}
where $C$ only depends on $g_0$.
\end{theorem}

We also prove a theorem that quantifies the bilipschitz constant of the identity map between $(S^4,g_0)$ and $S^4$ equipped with the standard metric. 
The solution to the Yamabe problem in this setting guarantees the existence of a metric in the conformal class of $g$ that has constant scalar curvature.  This metric is often called the Yamabe metric.
The following theorem states that the best bilipschitz constant between the Yamabe metric of a conformal class and the standard spherical metric depends on the $L^2$-norm of the Weyl tensor.
\begin{theorem}\label{thm-riccibilip}
Let $g_0$ be a Yamabe metric of constant and positive scalar curvature on $S^4$, let
\begin{align*}
    \beta = \int_{S^4}|W|^2 dv_0,
\end{align*}
and let
\begin{align*}
   \gamma = \int_{S^4}|B|^2 dv_0,
\end{align*}
where $W$ is the Weyl tensor and $B$ is the Bach tensor.
There exists a constant $\epsilon_2 > 0$ such that if $\beta < \epsilon_2$, then 
$(S^4,g_0)$ is bilipschitz equivalent to $(S^4,g_c)$, where $g_c$ is the standard metric for $S^4$. That is, for all $x,y \in \bR^4$ there exists $L_2 > 0$ such that
\begin{align*}
  \frac{1}{L_2}d_{g_c}(x,y) \le  d_{g_0}(x,y) \le L_2d_{g_c}(x,y),
\end{align*}
where $L_2 \le 1 + C'\gamma\beta$.
\end{theorem}

From these  theorems we directly deduce the main theorem of the paper.  
\begin{theorem}[Main Theorem]\label{thm-spherebilip}
Let $g$ be a smooth Riemannian metric on $S^4$ with a positive Yamabe constant and let $g_0$ be the Yamabe metric in the conformal class of $g$ with constant scalar curvature.  Additionally, suppose that
\begin{align*}
    \int_{S^4} Q dv \ge 0,
\end{align*}
where $Q$ is the $Q$-curvature for $g$.  There exists constants $\epsilon_1, \epsilon_2 > 0$ so that if
\begin{align*}
  \alpha = \int_{S^4} |Q_1e^{4w} - Q_0|dv_0 < \epsilon_1
\end{align*}
and
\begin{align*}
   \beta = \int_{S^4} |W_g|^2 dv_g < \epsilon_2,
\end{align*}
then there exists a bilipschitz map $f \colon (S^4,g) \to (S^4,g_c)$, where $g_c$ is the standard metric on $S^4$, such that 
the bilipschitz constant $L$ of $f$ depends on $\alpha$ and $\beta$ in the sense that
\begin{align*}
    L \le 1 + C_1\alpha +C_2\beta,
\end{align*}
where $C_1$ may depend on the conformal class of $g$ and $C_2$ only depends on the $L^2$-norm of the Bach tensor for $g_0$.
\end{theorem}

The result in the main theorem is related to the general question of bilipschitz uniformization of a metric space.  
What conditions are needed on a metric space for it to be bilipschitz equivalent to a model space (in our case the four-dimensional sphere)?
In the two-dimensional setting this problem has been studied extensively.
Bonk and Lang in \cite{bonklang} showed that $(\bR^2,g)$ is bilipschitz equivalent to $(\bR^2,|dx|^2)$ if the Gaussian curvature is integrable and the integral of the positive part of the Gaussian curvature is less than $2\pi$.  In two dimensions this is a sharp result as a counterexample can be constructed by taking an infinite cylinder with one hemispherical cap.  This theorem was first conjectured by Fu in \cite{fu}, where he showed a similar result with the added condition that the $L^1$-norm of the Gaussian curvature of $g$ is small.

Bonk, Heinonen and Saksman in \cite{bonkheinonensaksman} were able to extend the two dimensional result to all even dimensions.  They showed that $(\bR^4,e^{2w}|dx|^2)$ is bilipschitz to $(\bR^4,|dx|^2)$ whenever the $L^1$-norm of the $Q$-curvature for $e^{2w}$ is sufficiently small.
Their methods can be applied to the spherical setting and so a special case of Theorem \ref{thm-qcflowbilip}, when $g_1$ is conformally flat and $g_0$ is the standard metric on the sphere, follows from their work.

The results in \cite{bonkheinonensaksman} were improved upon in \cite{wang12} where Wang showed that if the $Q$-curvature is integrable and the positive part of the $Q$-curvature integrates to less than $2^{n-2}((n-2)/2)!\pi^{n/2}$, then the space is bilipschitz equivalent to the standard Euclidean space.
In both the two dimensional and higher dimensional cases the inequality is sharp.  An infinite cylinder with one hemispherical cap will give the equality case and such a manifold is not bilipschitz equivalent to the standard Euclidean space.
This method was also applied to the Heisenberg group by Austin in \cite{austin}, where the author constructed a quasiconformal map whose Jacobian was comparable to the logarithmic potential of a given measure with small total variation.

We also note that quantifying the best bilipschitz constant between two non-compact manifolds also quantifies the best constant in the isoperimetric inequality of the manifold.
In dimension two, the asymptotic isoperimetric inequality on a non-compact surface depends only on the integral of the Gaussian curvature.  This was shown by Finn in \cite{finn} and followed the work of Cohn-Vossen \cite{cohn-vossen} and Huber \cite{huber}.
Following the two-dimensional result, asymptotic isoperimetric inequalities in four dimensions  were shown to depend on the $L^1$-norm of the $Q$-curvature by Chang, Qing and Yang in \cite{changqingyang2} and \cite{changqingyang}.  This was also generalized by Wang, in \cite{wang15}.  Wang showed that the constant in the isoperimetric inequality only depends on the $L^1$-norm of the $Q$ curvature when the underlying manifold is $\bR^4$.

We adapt the methods in \cite{bonkheinonensaksman} to consider the non-conformally flat setting.
We study the dependence of the best bilipschitz constant between $(S^4,g)$ and $(S^4,g_c)$, where $g_c$ is the standard metric on $S^4$, on integral curvature quantities.  The dependence relies on the conformal invariant $L^2$-norm of the Weyl tensor and the $L^1$-norm of the $Q$-curvature of $g$.  Our main theorem can also be applied to certain settings on $\bR^4$. For example, if $\bR^4$ is equipped with a metric that is asymptotically flat, we expect our results to still hold.
For simplicity we work only on $S^4$ and we do not discuss the non-compact setting.
In the compact setting, the identity map between $(S^4,g_1)$ and $(S^4,g_0)$ is always bilipschitz.  So, opposed to the previous work on $\bR^4$, the question of the existence of a bilipschitz map is no longer relevant.  
Rather, we study the dependence of the best bilipschitz constant on the $Q$-curvature and Weyl tensor.

The bilipschitz map given by the main theorem is achieved as the composition of two maps. 
The first map is a bilipschitz map between the original metric and the Yamabe metric within its conformal class.  This map is constructed via a vector flow argument that follows the work in \cite{bonkheinonensaksman}.
The flow yields a quasiconformal map from $(S^4,g_1)$ to $(S^4,g_0)$; it is additionally bilipschitz as a map between $(S^4,g_1)$ and $(S^4,g_0)$.

The use of vector fields to construct quasiconformal mappings on $\bR^n$ was first developed by Reimann in \cite{reimann76}.  He showed that if the Ahlfors conformal strain tensor of a vector field $v$ satisfied
\begin{align*}
    \|Sv\|_\infty = \| (Dv + Dv^T)/2 - \operatorname{Tr}(Dv)\operatorname{Id}/n\|_\infty < c,
\end{align*}
then the flow along $v$ is quasiconformal with a quasiconformal constant that is less than $e^{nct}$ at time $t \ge 0$.
Pierzchalski in \cite{pierzchalski} partially generalized the work of Reimann to the Riemannian manifold setting.
Bonk, Heinonen and Saksman, in \cite{bonkheinonensaksman}, developed this further to show that the Jacobian of the flow maps corresponded to the divergence of $v$.  We discuss the quasiconformal theory in Section \ref{sec-qcflows}.
We also generalize the quasiconformal vector flow theory of the above authors to the more general setting of Riemannian metrics on $S^4$.

On a Riemannian manifold $(M,g)$, the $Q$-curvature is defined as
\begin{align*}
    Q = \frac{1}{12}\paren{-\Delta R + \frac{1}{4}R^2-3|E|^2},
\end{align*}
where $R$ is the scalar curvature and $E$ is the traceless Ricci tensor.  The $Q$-curvature satisfies a conformal transformation law similar to the Gaussian curvature for surfaces.  More precisely,
\begin{align}\label{eq-qtransformation}
    P_{g_0}w + 2 Q_{g_0}=2Q_ge^{4w},
\end{align}
whenever $g = e^{2w}g_0$.  Here, $P_{g_0}$ is the Paneitz operator and is defined by
\begin{align*}
    P = \Delta^2 + \operatorname{div}\paren{\frac{2}{3}R-2\operatorname{Ric}}d.
\end{align*}
If $g_0 = |dx|^2$ on $\bR^4$, then \eqref{eq-qtransformation} becomes
\begin{align*}
    \Delta^2w = 2Q_ge^{4w}.
\end{align*}
The function $w$, under the condition that $e^{2w}|dx|^2$ is a normal metric, can be expressed as a logarithmic potential
\begin{align*}
    w(x) = \frac{1}{8\pi^2}\int_{\bR^4} \log\frac{|y|}{|x-y|}Q_g(y)e^{4w(y)}dy + C,
\end{align*}
where $C$ is a constant (see \cite{changqingyang} for a discussion of normal metrics).
Using this observation, the authors of \cite{bonkheinonensaksman} defined a vector field $v$ so that
\begin{align*}
    \operatorname{div}v \sim \int_{\bR^4} \log\frac{|y|}{|x-y|}Q_g(y)e^{4w(y)}dy.
\end{align*}
If a map is given as the solution to the differential equation
\begin{align*}
    \frac{df(t,x)}{dt} = v(f(t,x)),
\end{align*}
then the Jacobian of $f$ will be comparable to $e^{4w}$.  The map $f$ was then shown to be quasiconformal as well, which implies that it is bilipschitz between $(\bR^4,e^{2w}|dx|^2)$ and the standard Euclidean space.

The crucial part of the argument in \cite{bonkheinonensaksman} is that the vector field arises from a logarithmic potential.  In order to adapt these methods to prove Theorem \ref{thm-qcflowbilip}, we employ the Green's function for the Paneitz operator as a substitute for the logarithmic potential.  
The Green's function for the Paneitz have been studied previously (see e.g., Gursky and Malchiodi \cite{gurskymalchiodi} and Hang and Yang \cite{hangyang}).
Whenever the Paneitz operator is invertible, the function $w$ can be expressed as
\begin{align*}
    w(x) = \int_{S^4} G(x,y)Q(y)e^{4w(y)}dv_{g_0}(y) +C,
\end{align*}
where $G(x,y)$ is the Green's function and $C$ is a constant.
We then show that the Green's function behaves similarly to a logarithmic potential, which allows us to use methods as in \cite{bonkheinonensaksman} to prove Theorem \ref{thm-qcflowbilip}.

Theorem \ref{thm-riccibilip} is a result of the fact that the Ricci flow on $(S^4,g_0)$ converges as $t \to \infty$ to the standard metric on the sphere whenever the Weyl tensor has a sufficiently small norm in $L^2$.
This was originally shown by Gursky in \cite{gursky94} for Yamabe metrics on $S^n$. It was shown by Chen in \cite{chen19} for asymptotically flat metrics on $\bR^n$.  

In the literature there are many results that assume some curvature pinching on $S^4$ or $\bR^4$ to imply that the Ricci flow converges to the standard structure (see e.g., \cite{chen19,gursky94,hamilton,huisken,margerin}).
Recently, Chang and Chen in \cite{changchen2021} proved quantitative estimates on the rate of convergence of the flow that only depend on the $L^2$-norm of the Weyl tensor.  
They also showed that the $L^2$-norm of certain integral curvature quantities are monotonically decreasing in time.  
These previous results give sufficient convergence for large time to the standard structure on the sphere. For short time, we introduce the $L^2$-norm of the Bach tensor to control the bilipschitz distortion of the flow.

The Weyl tensor arises from the orthogonal decomposition of the Riemannian tensor as
\begin{align*}
    \operatorname{Rm} = W + \frac{1}{2}E\KN g + \frac{1}{24}Rg\KN g,
\end{align*}
where $E = \Ric - \frac{1}{4}Rg$ is the traceless Ricci tensor.  The Weyl tensor is traceless and conformally invariant.
The Bach tensor arises as the gradient of the functional $g \mapsto \int |W|^2 dv$.  Since this functional is also conformally invariant it can be shown that the Bach tensor transforms as
\begin{align*}
    B_{e^{2w}g} = e^{-2w}B_{g}.
\end{align*}
In coordinates, $B$ is expressed as
\begin{align*}
    B_{ij} = \nabla^k\nabla^l W_{kijl} + \frac{1}{2}\Ric^{kl}W_{kijl}.
\end{align*}
The $L^1$-norm of the Bach tensor is a conformal invariant.  However, for $p > 1$, the $L^p$-norm of $B$ depends on the conformal factor by the transformation law for $B$.

The argument in \cite{changchen2021} uses the Chern-Gauss-Bonnet formula to show that if the Weyl tensor has small $L^2$-norm, then the $L^2$-norm of $|W| + |E| + |R-\overline{R}|$ is also controlled, where $\overline R$ is the average scalar curvature on $S^4$.  Recall that on $(S^4,g_0)$,
\begin{align*}
    16\pi^2 =  \int_M |W|^2 - \frac{1}{2}|E|^2 + \frac{1}{24}R^2 dv_{g_0}.
\end{align*}
If $g_0$ is a positive Yamabe metric, then
\begin{align*}
    \frac{1}{24}\int_{S^4}R^2 dv_{g_0} \le 16 \pi^2.
\end{align*}
Therefore, we have that $\|E\|_2 \le 2\|W\|_2$.
This, combined with parabolic-type estimates, gives exponential convergence of the metric to the standard structure on $S^4$ along the Ricci flow.

This estimate provides fast convergence for large time.  For short time we use the $L^2$-norm of the Bach tensor in order to show added integrability for $W$.  We show that the $L^4$-norm of $W$ and $E$ is controlled by the $L^2$-norms of both $W$ and $B$.  Added with the parabolic-type estimates for the Ricci flow, we prove sufficient bounds for small time to show Theorem \ref{thm-riccibilip}.

In Section \ref{sec-qcflows}, we provide the preliminary lemmas related to quasiconformal maps and the Green's function for the Paneitz operator needed to prove Theorem \ref{thm-qcflowbilip}.  This includes a discussion of quasiconformal flows on Riemannian manifolds.
In Section \ref{sec-proofqcflow}, we prove Theorem \ref{thm-qcflowbilip} using quasiconformal flows.  In Section \ref{sec-ricciproofprelim}, we provide the preliminary estimates of the norm of $W$ using the Bach tensor.  In Section \ref{sec-proofricci}, we prove Theorem \ref{thm-riccibilip}.

\subsection{Notation}
We attempt to follow the standard notation in this paper.  For a Riemannian metric $g$ on an orientable manifold $M$, we denote $d_g$ to be the distance associated with $g$ and $dv_g$ to be the volume form with respect to $g$.
The quantities $R$, $Q$, $\operatorname{Ric}$, $E$, $W$ and $B$ are the scalar curvature, $Q$-curvature, Ricci tensor, traceless Ricci tensor, Weyl tensor and Bach tensor respectively. 
The space $W^{k,p}(M)$ denotes the Sobolev space for real-valued functions on $M$ and for manifolds $M_1$ and $M_2$. The space $W^{k,p}(M_1,M_2)$ denotes the Sobolev space for maps $f \colon M_1 \to M_2$.  If $U \subset (M,g)$, then $|U|_g$ denotes the volume of $U$ with respect to $dv_g$.  The volume may be denoted as $|U|$ when the metric is clear from context.
Constants written as $C(A)$ denote that the constant $C(A)$ depends on $A$.  The statement $A \sim B$ is used to mean that there exists a constant $C$ so that
$    \frac{1}{C} B\le A \le CB$.

\section{Preliminaries}\label{sec-qcflows}
\subsection{Quasiconformal Maps and Quasiconformal Flows}
\begin{definition}
\emph{
Let $(M_1,g_1)$ and $(M_2,g_2)$ be Riemannian manifolds.  A continuous map $f \colon (M_1,g_1) \to (M_2,g_2)$ is $K$-\emph{quasiconformal} if $f$ is a homeomorphism, $f \in W^{1,n}_{\text{loc}}(M_1,M_2)$ and for almost every $x \in M_1$
\begin{align*}
    \|Df(x)\|^n \le K J_f(x),
\end{align*}
where $Df$ is the derivative of $f$ and $J_f = \det(Df)$. The function $J_f$ can also be defined as the Hodge star dual of the pullback of the volume form of $(M_2,g_2)$.}
\end{definition}

If $(M_1,g_1) = (M_2,g_2) = (\bR^n,|dx|^2)$, then Reimann in \cite{reimann76} characterized vector fields whose flow maps are quasiconformal.  
The \emph{Ahlfors conformal strain tensor} of a vector field $v \colon \bR^n \to \bR^n$ is defined as
\begin{align}\label{eq-ahlforsstraintensor}
    Sv := \frac{1}{2}(Dv^T + Dv) - \frac{1}{n}\operatorname{Tr}(Dv)\operatorname{Id}.
\end{align}
\begin{prop}[{\cite[Theorem 5]{reimann76}}]\label{prop-reimannflow}
Let $v\colon \bR^n \times [0,\infty) \to \bR^n$ be a vector field such that for every $t \in \bR$, $v_t$ is continuous, $v_t \in W^{1,1}_{\operatorname{loc}}(\bR^n)$, $v_t(x) = O(|x|\log |x|)$ for $x \to \infty$ and $\|Sv_t\|_\infty \le c$, where $c$ is independent of $t$.
Then the solutions to the differential equation
\begin{align}\label{eq-ode}
    \frac{d f_t}{d t} = v_t(f_t(x))\quad \text{and} \quad f_0(x) = x
\end{align}
exist, are unique and are $e^{nct}$-quasiconformal.
\end{prop}
Bonk, Heinonen and Saksman extended this proposition to characterize the behavior of $J_{f_t}$.
\begin{prop}[{\cite[Proposition 3.6]{bonkheinonensaksman}}]\label{prop-bonkjacobianflow}
Let $v$ be as in Proposition \ref{prop-reimannflow} and let $f_t$ be the solution to \eqref{eq-ode}.
Then
\begin{align}
    \log J_{f_t}(x) = \int_0^t \diver v_{t'}(f_{t'}(x))dt'
\end{align}
for almost every $x \in \bR^n$.
\end{prop}
The goal of this section is to prove a suitable version of these theorems for the case of closed and orientable Riemannian manifolds.
Let $(M,g)$ be a closed, orientable Riemannian manifold, we define the \textit{Ahlfors conformal strain tensor on $M$} as
\begin{align}\label{eq-ahlforsstraintensormfd}
    S_gv := \frac{1}{2}\cc L_v g - \frac{1}{n}(\diver_g v) g,
\end{align}
where $\cc L_v$ is the Lie derivative along $v$.
This definition was given by Pierzchalski in \cite{pierzchalski}.  
If a vector field $v$ satisfies $S_g v = 0$, then $v$ is a conformal Killing field.  The infinitesimal diffeomorphisms produced by flowing along $v$ will be conformal with respect to $g$.
Additionally, the operator $S_g$ satisfies the conformal invariance property,
\begin{align}\label{eq-conformalinvarianceSg}
    S_{e^{2w}g} v = e^{2w}S_g v.
\end{align}
This follows from a quick calculation. We can express $2S_{e^{2w}g}v$ in local coordinates as
\begin{align*}
    2S_{e^{2w}g} v &= \cc L_v(e^{2w})g + e^{2w}\cc L_vg - \frac{2}{n}\paren{ \sum_{i=1}^n \frac{\partial v_i}{\partial x_i} + v_i  \frac{\partial (\log (\det (e^{2w}g)^{1/2}))}{\partial x_i} }e^{2w}g\\
    &=L_v(e^{2w})g - \paren{2\sum_{i=1}^n v_i\frac{\partial w}{\partial x_i}}e^{2w}g + e^{2w}\cc L_vg - \frac{2}{n}\diver_g(v) e^{2w}g.
\end{align*}
The first two terms cancel and the second two terms correspond to $2e^{2w}S_g v$.  The conformal invariance in \eqref{eq-conformalinvarianceSg} immediately gives that
\begin{align}\label{eq-normSinvariant}
    |S_{e^{2w}g}v|_{e^{2w}g} = |S_g v|_g
\end{align}
and so the supremum norm of $S_gv$ is invariant under any conformal change to the metric $g$.

Pierzchalski showed the following proposition.
\begin{prop}[{\cite[Theorem 3]{pierzchalski}}]\label{prop-smoothodeexistence}
Suppose $(M,g)$ is a closed and orientable Riemannian manifold.  Let $v \colon M \times [0,\infty)$ be a smooth vector field satisfying $\|S_gv_t\|_\infty \le c$, where $c$ is independent of $t$.  Then the solutions to the differential equation
\begin{align*}
     \frac{d f_t}{d t} = v_t(f_t(x))\quad \text{and} \quad f_0(x) = x
\end{align*}
exist, are unique and are $e^{nct}$-quasiconformal.
\end{prop}
Pierzchalski proved the deformation theorem only for smooth vector fields.
We will now show that the result is still true when the vector fields are only in $W^{1,1}(S^n)$.
\begin{prop}\label{prop-qcflowmfd}
Let $v \colon S^n \times[0,\infty) \to \bR^n$ be a vector field such that for every $t\in [0,\infty)$, $v_t \in W^{1,1}(S^n)$
and $\|S_gv_t\|_\infty \le c$, where $c$ is independent of $t$ and $S$ is defined in terms of a smooth Riemannian metric $g$ on $S^n$.  Then the solutions to the differential equation
\begin{align}\label{eq-mfdode}
    \frac{d f_t}{dt} = v_t(f_t(x)) \quad \text{and} \quad f_0(x) = x
\end{align}
exist, are unique and are $e^{nct}$-quasiconformal.
\end{prop}
We first show the following lemma.
\begin{lemma}\label{lemma-Scomparable}
Let $S_{g_c}$ and $S_g$ be the Ahlfors conformal strain tensors for the metrics $g_c$ and $g$, where $g_c$ is the standard metric on the sphere.
Then for any vector field $v \colon S^n \times [0,T] \to \bR^n$ with measurable derivatives,
\begin{align*}
    \|S_{g_c}v\|_\infty \le C(g)(\|S_g v\|_\infty +  \|v\|_\infty).
\end{align*}
\end{lemma}
\begin{proof}
Fix $x \in S^n$. By \eqref{eq-normSinvariant}, it suffices to consider $Sv$ instead of $S_{g_c}v$, where $S$ is the tensor taken with respect to the Euclidean metric.
We can express Equation \eqref{eq-ahlforsstraintensormfd} in coordinates, where $x = 0$ and $g(0) = \operatorname{Id}$,
\begin{align*}
    2(S_{g}v)_{ij} = (Dv g)_{ij} + (Dv g)^T_{ij} - \frac{2}{n} \operatorname{Tr}(Dv) g_{ij} + v^k \frac{\partial g_{ij}}{\partial x^k} +v^k\frac{1}{(\det g)^{1/2}} \frac{\partial (\det g)^{1/2}}{\partial x^k}g_{ij}.
\end{align*}
So at $x = 0$,
\begin{align*}
    (Dv)_{ij} + (Dv)^T_{ij} - \frac{2}{n} \operatorname{Tr}(Dv) \delta_{ij} = 2(S_{g}v)_{ij} -\paren{v^k \frac{\partial g_{ij}}{\partial x^k} +v^k\frac{1}{(\det g)^{1/2}} \frac{\partial (\det g)^{1/2}}{\partial x^k}\delta_{ij}}
\end{align*}
and hence
\begin{align*}
    | Dv  + (Dv )^T - \frac{2}{n} \operatorname{Tr}(Dv)| \le  2|S_g v| + C(g) |v|.
\end{align*}
Since $x$ was an arbitrary point,
$|S_{g_c}v| = |Sv(x)|_g \le |S_gv(x)| + C(g)|v|$, for all $ x \in S^n$.
\end{proof}

\begin{proof}[Proof of Proposition \ref{prop-qcflowmfd}]
By \eqref{eq-normSinvariant}, the norm $\|S\cdot\|_\infty$ is conformally invariant. By Proposition \ref{prop-reimannflow}, the statement in Proposition \ref{prop-qcflowmfd} also holds when $g = g_c$. 
For a general metric $g$, by Lemma \ref{lemma-Scomparable}, we have that
$\|S_{g_c} v \|_\infty$ is bounded by $\|S_g v\|_\infty$ and $\|v\|_\infty$ for all time and so the solutions to Equation \eqref{eq-mfdode} exist, are unique and are quasiconformal with respect to $g_c$.

It remains to show that the solutions are $e^{nct}$-quasiconformal.
We now derive some formulas for $f_t$ and its derivatives.  We have that
\begin{align*}
    f_t(x) = x + tv + o(t)
\end{align*}
and, for $x \in S^n$,
\begin{align*}
    g(f_t(x)) = g(x) + t Dg(x)v(x) + o(t).
\end{align*}
By \cite[Lemma 3.4]{bonkheinonensaksman}, if fix $T > 0$ is fixed, for almost every $x \in S^n$ and every $t \in [0,T]$,
\begin{align}\label{eq-Dfintegralformula}
    Df_t(x) = I_n + \int_0^t Dv(f_{t'}(x)) Df_{t'}(x)dt' 
\end{align}
and so
\begin{align*}
    Df_t = I_n + t Dv + o(t).
\end{align*}
We can now compute a first order formula for the Jacobian of $f_t$,
\begin{align*}
   J_{f_t} =  *f^*_t(\vol_g) = 1 + t \diver_g v + o(t).
\end{align*}
Let $G_t(X,Y) =J_{f_t}^{-2/n} g(Df_t X, Df_t Y)$. If we write $g(X,Y)$ in matrix notation as $X^T g Y$, then 
\begin{align*}
    G_t(X,Y) &= (1 + t \diver_g v +o(t))^{-2/n}[(I_n + tDv + o(t))^TX^T](g + t Dg v + o(t))[(I_n + tDv + o(t))Y] \\
    &=(1 + t \diver_g v +o(t))^{-2/n}(g(X,Y) + t (g(Dv X,Y) + g(X,Dv Y) + Dg v(X,Y)) +o(t)).
\end{align*}
By the definition of $\cc L_g v$ and expanding the first term,
\begin{align*}
     G_t(X,Y) &= (1- t \frac{2}{n} \diver_g v  + o(t)) (g(X,Y) + t \cc L_vg(X,Y) +o(t))  \\
    &= g(X,Y) + 2tS_gv(X,Y) + o(t).
\end{align*}
So the derivative of $G_t$ at $t = 0$ can be computed,
\begin{align}\label{eq-derivGat0}
    \frac{d}{dt}(G_t(X,Y))|_{t = 0} = 2S_gv(X,Y).
\end{align}
Following \cite[Corollary 1]{pierzchalski}, we now show that
\begin{align}\label{eq-derivativejacobian}
    \frac{d}{dt}(G_t(X,Y))|_{t = s} = 2J_{f_s}^{-2/n}S_gv(Df_sX,Df_sY).
\end{align}
To see this note that
\begin{align*}
    G_{t+s}(X,Y) = J_{\tilde f_t}^{-2/n} J_{f_s}^{-2/n}2S_gv(D\tilde f_t(Df_sX),D\tilde f_t(Df_sY)),
\end{align*}
where $\tilde f$ is the flow map along the vector field $v_s(x,t) = v(x,s+t)$.
So
\begin{align*}
    \frac{d}{dt}(G_t(X,Y))|_{t = s} &= \lim_{t\to 0} \frac{1}{t}(G_{t+s}(X,Y) - G_s(X,Y)) \\
    &= \lim_{t\to 0} \frac{1}{t}(J_{\tilde f_t}^{-2/n} J_{f_s}^{-2/n}g(D\tilde f_t(Df_sX),D\tilde f_t(Df_sY)) - J_{f_s}^{-2/n}g(Df_sX,Df_sY)) \\
    &= J_{f_s}^{-2/n}\frac{d}{dt}  G_t^{v_s}(Df_sX,Df_sY)|_{t=0},
\end{align*}
where $G^{v_s}$ is defined as $G$ except the mapping used comes from the vector field $v_s$.
Applying \eqref{eq-derivGat0} yields \eqref{eq-derivativejacobian}.
Therefore,
\begin{align*}
    \abs{\frac{d}{dt}(G_t(X,Y))|_{t = s}} &= 2\abs{J_{f_s}^{-2/n}S_gv(Df_sX,Df_sY)} \\
    &\le 2 \|S_g v\|_\infty \abs{J_{f_s}^{-2/n}g(Df_sX,Df_sY)} \\
    & \le 2c G_s(X,Y).
\end{align*}
Integrating the bound for $G$, we see that
\begin{align*}
    g(Df_t X, Df_t Y) \le e^{2ct} J_{f_t}^{2/n}
\end{align*}
and so $f_t$ is $e^{nct}$-quasiconformal.
\end{proof}
The following proposition is a generalization of \cite[Proposition 3.6]{bonkheinonensaksman} to the Riemannian setting.  The proof is essentially the same.
\begin{prop}\label{prop-jacobianflowmfd}
Let $v$ be as in Proposition \ref{prop-qcflowmfd} and let $f_t$ be the solution to \eqref{eq-mfdode}.  If $\vol$ is the volume form for $g$ and $*$ is the Hodge star operator, then for almost every $x \in S^n$
\begin{align*}
    \log(*f_t^*(\vol)) = \int_0^t \diver v(f_{t'},t')dt'.
\end{align*}
\end{prop}
\begin{proof}
Let $J_{f_t} = *f_t^*(\vol)$ and
\begin{align*}
    A(t) := J_{f_t} \exp \paren{-\int_0^t \diver v(f_{t'}(x),t')dt'}
\end{align*}
Due to Lemma \ref{lemma-Scomparable}, the maps $f_t$ are quasiconformal with respect to the standard metric on $S^n$.  This allows us to apply \cite[Lemma 3.4]{bonkheinonensaksman}, which asserts that $A$ is differentiable almost everywhere.
By \eqref{eq-Dfintegralformula} and since $Df_t$ is invertible almost everywhere,
\begin{align*}
    \frac{d}{dt}J_{f_t}\bigg|_{t = s} &= J_{f_s} \operatorname{Tr}_g\paren{\frac{d}{dt}Df_t\bigg|_{t=s}Df_{s}^{-1}}\\
    &= J_{f_s} \operatorname{Tr}_g(Dv_s(f_s)).
\end{align*}
So
\begin{align*}
A'(s) &=  J_{f_s}\operatorname{Tr}_g(Dv(f_s,s))\exp \paren{-\int_0^s \diver v_{t'}(f_{t'}(x),t')dt'} 
- J_{f_s}\exp \paren{-\int_0^s \diver v(f_{t'}(x),t;)dt'}\diver_g v_s(f_s) \\
&= 0.
\end{align*}
\end{proof}

We will also need the following general facts from the theory of quasiconformal maps.

\begin{prop}\label{prop-qcmapnormalfamily}
Suppose that $W$ is a family of $K$-quasiconformal maps from $(S^n,g)$ to itself.  There exists a sequence in $W$ and a point $p \in S^n$ such that the sequence converges locally uniformly to either a constant function on $S^n\setminus \{p\}$ or a $K$-quasiconformal map.
\end{prop}
\begin{proof}
This theorem holds in the case when the metric on $S^n$ is the standard one, $g_c$, and the proof can be found in \cite[Theorem 6.6.26]{gehringmartinpalka}.
It suffices to show that if $f \colon (S^n,g) \to (S^n,g)$ is $K$-quasiconformal with respect to $g$, then $f$ is $K'$-quasiconformal with respect to $g_c$, where $K'$ only depends on $g$ and $K$.

Since $S^n$ is compact,
\begin{align*}
    \frac{1}{C}\|Df\|_{g_c} \le \|Df\|_g \le C \|Df\|_{g_c} \quad \text{and} \quad \frac{1}{C}f^*\vol_{g_c} \le f^*\vol_g \le C f^*\vol_{g_c},
\end{align*}
where the constant in the inequalities depends only on $g$.
The defining inequality for quasiconformal maps shows that $f$ is quasiconformal with respect to $g_c$ as well.  Finally, a limiting map will be $K$-quasiconformal (not only $K'$ with respect to $g$) since the uniform limit of $K$-quasiconformal maps is $K$-quasiconformal (see e.g., \cite[Corollary 10.30]{heinonenbook}).
\end{proof}

\begin{prop}\label{prop-qcandcomparablejacobianbilip}
Let $(M,g)$ be an $n$-dimensional Riemannian manifold and let $f \colon (M,g) \to (M,g) $ be a $K$-quasiconformal map. Suppose that the Jacobian of $f$, $J_f$ satisfies
\begin{align*}
    \frac{1}{C}e^{nw} \le J_f \le C e^{nw},
\end{align*}
then $f$ is a bilipschitz map from $(M,e^{2w}g)$ to $(M,g)$ with constant $CK$.
\end{prop}
\begin{proof}
We first show that $|f(B)|_g$ is comparable to $|B|_{e^{2w}g}$ for all balls in $M$.
\begin{align*}
    |f(B)|_g &= \int_B f^*\vol_g = \int_B J_f \vol_g\\
    &\le C\int_B e^{nw} \vol_g\\
    &= C|B|_{e^{2w}g}.
\end{align*}
Since $f$ is quasiconformal we have that $\|Df\|^n \le K J_f$ and so $Df(x)$ is bounded almost everywhere as a map between $(T_xM,e^{2w}g)$ and $(T_{f(x)}M,g)$ by $CK$.  This and the fact that $f \in W^{1,n}_{\text{loc}}(M,M)$ gives that $f$ is Lipschitz.
The same argument for the inverse gives that $f$ is bilipschitz.

\end{proof}

\begin{prop}\label{prop-jacobiancomparability}
Let $f \colon (S^n,g) \to (S^n,g)$ be a $K$-quasiconformal map. For any $x,y \in S^n$ and $0 < c < 1$,
\begin{align*}
    \frac{1}{C(g)^K}d_g(f(x),f(y)) \le  \abs{f(B(x,cd_g(x,y))}^{1/n} \le C(g)^K d_g(f(x),f(y)),
\end{align*}
where $C(g)$ depends on $c$.
Additionally, for any ball $B \subset M$,
\begin{align*}
    |f(2B)| \le C(g)^K|f(B)|.
\end{align*}
\end{prop}
\begin{proof}
These comparability relations are both well-known facts for quasiconformal maps on the standard sphere and as used in the proof of Proposition \ref{prop-qcmapnormalfamily}, $f$ is quasiconformal with respect to $g_c$ quantitatively.
Since we need explicit bounds for our dependencies we provide a proof.

If $f$ is $K$-quasiconformal with respect to $g$, then $f$ is $K' = C(g)K$ quasiconformal with respect to $g_c$, the standard metric on $S^n$.
So
\begin{align*}
    d_{g_c}(f(x),f(y)) &\le \sup_{\{z:d_{g_c}(x,z) = d_{g_c}(x,y)\}} d_{g_c}(f(x),f(z)) \\
    &=\frac{\sup_{\{z:d_{g_c}(x,z) = d_{g_c}(x,y)\}} d_{g_c}(f(x),f(z))}{\inf_{\{z:d_{g_c}(x,z) = cd_{g_c}(x,y)\}} d_{g_c}(f(x),f(z))}\inf_{\{z:d_{g_c}(x,z) = cd_{g_c}(x,y)\}} d_g(f(x),f(z)) \\
    &\le  C\eta_{K'}(1/c) \abs{f(B(x,cd_g(x,y))}^{1/n}_{g_c},
\end{align*}
where $\eta_{K'}$ is the quasisymmetry function for $f$ and depends only on $K$.  The function $\eta_{K'}$ is defined as follows: For any $K'$-quasiconformal map $f \colon (S^n,g_c) \to (S^n,g_c)$, there exists an increasing homeomorphism $\eta_{K'} \colon [0,\infty)\to [0,\infty)$ that depends only on $K'$ so that
\begin{align*}
    \frac{d_{g_c}(f(x),f(y))}{d_{g_c}(f(x),f(z))} \le \eta \bigg ( \frac{d_{g_c}(x,y)}{d_{g_c}(x,z)}\bigg ),
\end{align*}
for all $x,y,z \in S^n$.  In fact, by \cite[Lemma 2.1]{bonkheinonensaksman},
\begin{align*}
    \eta_{K'}(s) \le 4^{K'}e^{2K'(n-1)}(1+s)^{K'}.
\end{align*}
Therefore,
\begin{align*}
    d_{g_c}(f(x),f(y)) &\le C^{K'} \abs{f(B(x,d_g(x,y))}^{1/n}_{g_c} \\
    & \le C(g)^{K} \abs{f(B(x,d_g(x,y))}^{1/n},
\end{align*}
where $C(g) > 1$ depends only on $g$.
The second inequality in the first relation of the proposition is proven similarly.

For the second relation,
\begin{align*}
    \frac{|f(2B)|_g}{|f(B)|_g} &\le C(g) \frac{|f(2B)|_{g_c}}{|f(B)|_{g_c}} \\
    &\le  C(g) \frac{\sup_{y \in 2\partial B} d_{g_c}(f(x),g(y))^n}{\inf_{z \in \partial B} d_{g_c}(f(x),g(z))^n} \\
    &\le C(g)  \eta_{K'}(2)^n,
\end{align*}
Therefore
\begin{align*}
    \frac{|f(2B)|_g}{|f(B)|_g}  \le C(g)^K,
\end{align*}
where $C(g) > 1$ depends only on $g$.

\end{proof}

\begin{remark}
\emph{If  $g$ is bilipschitz equivalent to $g_c$, the standard metric on the sphere $g$, then the constant $C(g)$ will only depend on the bilipschitz constant.  If $g$ is a Yamabe metric on $S^4$ and the $L^2$-norm of its Weyl tensor is sufficiently small, then $C(g)$ depends only on the $L^2$-norms of the Weyl and Bach tensors by Theorem \ref{thm-riccibilip}.  This will be the setting when we apply this proposition below in Section \ref{sec-proofqcflow}.}
\end{remark}

\subsection{The Green's Function}
In the proof for Theorem \ref{thm-qcflowbilip} we use the Green's function of the Paneitz operator.
Gursky in \cite[Theorem A]{gursky99} showed that the Green's function for the Paneitz operator exists given the hypotheses in Theorem \ref{thm-qcflowbilip}.  The following proposition gives a formula for the Green's function near a pole.  For a more detailed study of the asymptotic behavior of the Green's function near its poles we refer the reader to \cite{hangyang}.

\begin{prop}\label{prop-greenfunctionexpansion}
If $(M,g)$ is a smooth compact $4$-dimensional Riemannian manifold such that
\begin{align*}
    \int_M Q_g \ge 0
\end{align*}
and
$Y([g]) \ge 0$,
then the Green's function for the Paneitz operator $G(x,y)$ exists and can be expressed as
\begin{align}\label{eq-greensfunctionexpansion}
    G(x,y) = \frac{1}{8\pi^2} \log \frac{1}{d_{g}(x,y)} + h(x,y),
\end{align}
where $h(x,y)$ is smooth and bounded independently of $x$ and $y$.
\end{prop}
In order to prove Proposition \ref{prop-greenfunctionexpansion}, we quote the following lemma.
\begin{lemma}[{\cite[Lemma 2.8]{gurskymalchiodi}}]\label{lemma-radial}
Let $(M,g)$ be a smooth compact $4$-dimensional Riemannian manifold such that
\begin{align*}
    \int_M Q_g \ge 0
\end{align*}
and
$Y([g]) \ge 0$.
If $u$ is a radial function around a point $o \in M$ and $\tilde g$ is the metric for the conformal normal coordinates around $o$, then
\begin{align*}
    P_{\tilde g} u = \Delta_g^2 u + \frac{1}{12}\nabla_k\nabla_l R(o) x^kx^l \cc D(u) +O(r) u + O(r^2)|u'|+ O(r^3)|u''|+ O(r^4)|u'''|,
\end{align*}
where
\begin{align*}
    \cc  D(u) = \frac{2u'}{r}- 2u''
\end{align*}
and $r = d(o,x)$.
\end{lemma}
\begin{proof}[Proof of Proposition \ref{prop-greenfunctionexpansion}]
The existence of $G$ follows from \cite[Theorem A]{gursky99}.
We apply Lemma \ref{lemma-radial} to $u(r) = \frac{1}{8 \pi^2}\log \frac{1}{r}$.  If $r = d_{\tilde g}(x,y)$,
\begin{align*}
    P_{\tilde g}u = \delta_x(y) +\frac{1}{12} \nabla_k\nabla_l R(o)x^kx^l \cc D\paren{\frac{1}{8\pi^2}\log \frac{1}{r}} + O(r\log r).
\end{align*}
and
\begin{align*}
    \cc D(\log \frac{1}{r}) = O(r^{-2}).
\end{align*}
So
\begin{align*}
     P_{\tilde g}\log \frac{1}{r} = \delta_x(y) + O(1).
\end{align*}
Therefore
\begin{align*}
    P_{\tilde g} \bigg ( G_{\tilde g}(x,y) - \log \frac{1}{r}) = O(1)
\end{align*}
and by elliptic regularity for $P_{\tilde g}$, $ G_{\tilde g}(x,y) - \log \frac{1}{r} \in W^{4,p}_{\text{loc}}$ for all $p > 1$.  This implies that the difference is in $C^1$.

To finish the proof we must show that the same estimate holds for $G$, i.e., the Green's function for the original metric.  Recall that if $\tilde g = e^{2w} g$, then $P_{\tilde g} = e^{-4w}P_g$.
We know then that $e^{-4w}P_{g}(G_{\tilde g} \phi) =  \phi$ for $\phi \in C_c^\infty$.  So 
\begin{align*}
    P_g(G_{\tilde g} (e^{-4w}\phi)) = \phi.
\end{align*}
If we write this in integral form,
\begin{align*}
    \phi(x) &= P_g \int G_{\tilde g}(x,y) \phi(y) e^{-4w(y)} dv_{\tilde g}(y)\\
    &= P_g \int G_{\tilde g}(x,y) \phi(y) dv_g(y).
\end{align*}
Therefore $\tilde G - G \in \ker P_g$ and must be a constant $C(\tilde g)$.  
So 
\begin{align*}
    G(x,y) = \frac{1}{8\pi^2}\log\frac{1}{d_{\tilde g}(x,y)} + h(x,y),
\end{align*}
where $h$ is bounded in $y$ for a fixed $x \in M$.  
Since $\tilde g$ is the conformal normal metric, the construction of which depends only on $g$, we can replace  $d_{\tilde g}(x,y)$ with $ d_g(x,y)$, where the remainder term is absorbed by $h(x,y)$.
Since $M$ is compact, we can cover $M$ with finitely many charts where we have conformal normal coordinates.  On each of these chart, $G(x,y)$ has the desired form and is bounded in $y$ for a fixed $x$.

Finally, to show that $h$ is bounded in terms of $x$ as well, it suffices to show that $h$ is continuous in $x$ and $y$.  Let $x_1,x_2,y_1$ and $y_2$ be points in $M$.  Then
\begin{align*}
    |h(x_1,y_1) - h(x_2,y_2)| &\le |h(x_1,y_1) - h(x_1,y_2)| + |h(x_1,y_2) - h(x_2,y_2)|\\
    &=|h(x_1,y_1) - h(x_1,y_2)| + |h(y_2,x_1) - h(y_2,x_2)|,
\end{align*}
by the symmetry of $G$ in $x$ and $y$.  Since $h(x_1,\cdot)$ and $h(y_2,\cdot)$ are both continuous, we have that $h$ is continuous.  Thus, $\sup_{(x,y) \in M\times M} |h(x,y)|$ is uniformly bounded and the proposition is proved.
\end{proof}

\section{Proof of Theorem \ref{thm-qcflowbilip}}\label{sec-proofqcflow}
We recall the theorem here for convenience.  The proof follows the structure of the argument in \cite{bonkheinonensaksman}.
\begin{named}{Theorem \ref{thm-qcflowbilip}}
Let $g_1$ be a metric so that $Y([g_1]) \ge 0$.  Let $g_0$ be a metric in the conformal class of $g_1$.
Suppose additionally that
\begin{align}
    \int_{S^4}Q_1 dv_1 \ge 0,
\end{align}
where $Q_1$ is the $Q$-curvature of $g_1$.
Let $w \colon S^4 \to \bR$ be the smooth function such that $g_1 = e^{2w}g_0$. 
There exists $\epsilon_1 > 0$ such that if
\begin{align*}
    \alpha = \int_{S^4} |Q_1e^{4w} - Q_0|dv_0 < \epsilon_1,
\end{align*}
then there exists a quasiconformal and bilipschitz map $f \colon (S^4,g_1) \to (S^4,g_0)$ whose bilipschitz constant depends only on $\alpha$ and $g_0$.  More precisely, the bilipschitz constant $L_1$ of $f$ satisfies
\begin{align*}
    L_1 \le 1 + C \alpha,
\end{align*}
where $C$ only depends on $g_0$.
\end{named}

In order to prove this theorem we will need a series of lemmas.  In the following $\Phi \colon (S^4,g_0) \to (S^4,g_0)$ will denote an arbitrary $K$-quasiconformal map.

For a given $K$-quasiconformal map $\Phi\colon (S^4,g_0) \to (S^4,g_0)$, we will need a smooth distance corresponding to $d_{g_0}(\Phi(x),\Phi(y))$ for $x,y \in S^4$. 
Define a bump function $\psi \colon \bR \to [0,\infty)$ so that
\begin{align*}
    \psi(x) &= 1 \text{ on } |x| < \frac{1}{4} \\
    \psi(x) &= 0 \text{ on } |x| > \frac{1}{2}. \\
\end{align*}
Define $\rho \colon S^4\times S^4 \to \bR$ as
\begin{align*}
    \rho(x,y) :=
    \biggl (\int_{S^4} \psi\biggl (\frac{d_{g_0}(x,z)}{d_{g_0}(x,y)}\biggr) J_\Phi(z) dv_0(z)\biggr) ^{1/4},
\end{align*}
where $dv_0$ is the volume form for $g_0$.
The following lemma shows that the function $\rho$ acts as a ``distance" function with respect to $\Phi$.
\begin{lemma}\label{lemma-smoothdistance}
The function $\rho$ satisfies
\begin{itemize}
    \item[(i)] $
    \frac{1}{C(g_0)^K} d_{g_0}(\Phi(x),\Phi(y)) \le \rho(x,y) \le  C(g_0)^K d_{g_0}(\Phi(x),\Phi(y))$ and 
\item[(ii)]
$
    \frac{D_x \rho(x,y)}{\rho(x,y)} \le \frac{C(g_0)^K}{d_{g_0}(x,y)}.
$
\end{itemize}
\end{lemma}
\begin{proof}
To see (i) note that
\begin{align*}
    \rho(x,y) \le \abs{\Phi\paren{B\paren{x,\frac{d_{g_0}(x,y)}{2}}}}_{g_0}^{1/4} \le C(g_0)^K d_{g_0}(\Phi(x),\Phi(y))
\end{align*}
and
\begin{align*}
    \rho(x,y) \ge \abs{\Phi\paren{B\paren{x,\frac{d_{g_0}(x,y)}{4}}}}_{g_0}^{1/4}\ge \frac{1}{C(g_0)^K}d_{g_0}(\Phi(x),\Phi(y)),
\end{align*}
where the second inequality in both lines follows from Proposition \ref{prop-jacobiancomparability}.
To show (ii),
we compute that

\begin{align*}
    D_x \psi\bigg ( \frac{d_{g_0}(x,z)}{d_{g_0}(x,y)}\bigg ) = (D_x \psi)\bigg ( \frac{d_{g_0}(x,z)}{d_{g_0}(x,y)}\bigg ) D_x \bigg ( \frac{d_{g_0}(x,z)}{d_{g_0}(x,y)}\bigg ).
\end{align*}
\begin{align*}
   D_x \bigg ( \frac{d_{g_0}(x,z)}{d_{g_0}(x,y)}\bigg ) = \frac{1}{d_{g_0}(x,y)}D_x(d_{g_0}(x,z)) - \frac{d_{g_0}(x,z)}{d_{g_0}(x,y)^2} D_x(d_{g_0}(x,y)).
\end{align*}
The terms $|D_x(d_{g_0}(x,y))|$ and $|D_x(d_{g_0}(x,z))|$ are bounded on $S^4$.
When $z \notin B(x,d_{g_0}(x,y))$, then the $\psi$ in the integral is $0$. 
If $z \in B(x,d_{g_0}(x,y))$, then $d_{g_0}(x,z) \le d_{g_0}(x,y)$. So
\begin{align*}
     \abs{D_x \bigg ( \frac{d_{g_0}(x,z)}{d_{g_0}(x,y)}\bigg ) }  \le \frac{C(g_0)}{d_{g_0}(x,y)}.
\end{align*}
Therefore,
\begin{align}\label{eq-psibound}
    \abs{D_x \psi \bigg ( \frac{d_{g_0}(x,z)}{d_{g_0}(x,y)} \bigg )}\le  \frac{C(g_0)}{d_{g_0}(x,y)}.
\end{align}
We can now bound the left hand side of (ii).
\begin{align*}
     \frac{D_x \rho(x,y)}{\rho(x,y)} &= \frac{1}{4}   \paren{\int_{S^4} D_x\psi\biggl (\frac{d_{g_0}(x,z)}{d_{g_0}(x,y)}\biggr) J_\Phi(z) dv_0(z)} \bigg / \paren{\int_{S^4} \psi\biggl (\frac{d_{g_0}(x,z)}{d_{g_0}(x,y)}\biggr) J_\Phi(z) dv_0(z)} \\
     &\le \frac{C(g_0)}{d_{g_0}(x,y)} \frac{\paren{\int_{B(x,d_{g_0}(x,y)/2)} J_\Phi(z) dv_0(z)}}{\paren{\int_{B(x,d_{g_0}(x,y)/4} J_\Phi(z) dv_0(z)}},
\end{align*}
where in the second inequality we used the support of $\psi$ and \eqref{eq-psibound}.
By Proposition \ref{prop-jacobiancomparability},
\begin{align*}
    \frac{D_x \rho(x,y)}{\rho(x,y)}  \le  \frac{C(g_0)^K}{d_{g_0}(x,y)}.
\end{align*}
\end{proof}

\begin{lemma}\label{lemma-potentialvectorfield}
For every $y \in S^4$, there exists a vector field $V(x,y)$ on $S^4$ such that
\begin{itemize}
    \item[(i)] $D_x V$ exists for all $y \in S^4$ and is integrable in $x$ and in $(x,y)$.
    
    \item[(ii)] $|\operatorname{div}_{g_0} V(x,y) - 4G(\Phi(x),y)| \le C(g_0,K)$, where the constant depends only on the metric $g_0$ and the quasiconformal constant $K$ of $\Phi$.
    
    \item[(iii)] $\|SV\|_\infty \le C(g_0)^K$, where $C(g_0)$ depends only on $g_0$.
\end{itemize}
\end{lemma}
Note that if $g_0$ locally admits a conformal Killing field $V'$ and $\Phi(x) = x$, then $V(x,y) = G(x,y)(V'(x)-V'(y))$.
\begin{proof}
Recall that by Proposition \ref{prop-greenfunctionexpansion},
\begin{align*}
    G(x,y) = \frac{1}{8\pi^2} \log \frac{1}{d_{g_0}(x,y)} + h(x,y),
\end{align*}
where $h$ is smooth and bounded on $S^4$.
Define $\tilde G(x,y)$ as
\begin{align*}
    \tilde G(x,y) = \frac{1}{8\pi^2} \log \frac{1}{\rho(x,\Phi^{-1}(y))} + h(x,y).
\end{align*}
By Lemma \ref{lemma-smoothdistance}, (i),  we have that
\begin{align}\label{eq-modifiedgreenclosetogreen}
    |\tilde G(x,y) - G(\Phi(x),y)| \le C(g_0,K)
\end{align}
and by Lemma \ref{lemma-smoothdistance}, (ii),
\begin{align}\label{eq-phigreenfunctionderivbound}
    |D_x \tilde G(x,y)|  \le \frac{C(g_0)^K}{d_{g_0}(x,\Phi^{-1}(y))}.
\end{align}
Define
\begin{align}\label{eq-Vdefinition}
    V(x,y) := \tilde G(x,y) \psi\paren{\frac{d_{g_0}(x,\Phi^{-1}(y))}{r}} (x - \Phi^{-1}(y)),
\end{align}
where $r$ is the injectivity radius for $g_0$ and $\psi$ is a non-negative smooth function on $\bb R$ such that $\psi = 1$ on $|x| < 1/4$ and $\psi = 0$ on $|x| > 1/2$.  The difference, $x-\Phi^{-1}(y)$, is taken in normal coordinates around $\Phi^{-1}(y)$, which is justified by our choice of $r$.
We compute,
\begin{align}\label{eq-DV}
\begin{split}
    D_x V(x,y) = & D_x \tilde G(x,y)(x-\Phi^{-1}(y)) \psi\paren{\frac{d_{g_0}(x,\Phi^{-1}(y))}{r}} \\
    & + \tilde G(x,y) (x-\Phi^{-1}(y)) D_x\psi\paren{\frac{d_{g_0}(x,\Phi^{-1}(y))}{r}} \\
    &+ \tilde G(x,y) \psi\paren{\frac{d_{g_0}(x,\Phi^{-1}(y))}{r}}\operatorname{Id}.
    \end{split}
\end{align}
By Proposition \ref{prop-greenfunctionexpansion}, $G$ and $D_x G(x,y) (x-y)$ are in $L^1(S^4)$ and we see that (i) is satisfied.

We next compute $\operatorname{div}_{g_0} V$.
\begin{align*}
    \operatorname{div}_{g_0} V = &D_x \tilde G(x,y)\psi\paren{\frac{d_{g_0}(x,\Phi^{-1}(y))}{r}}(x-\Phi^{-1}(y))\\
    &+\tilde G(x,y)D_x \psi\paren{\frac{d_{g_0}(x,\Phi^{-1}(y))}{r}} (x-\Phi^{-1}(y)) \\
    &+  \tilde G(x,y)\psi\paren{\frac{d_{g_0}(x,\Phi^{-1}(y))}{r}}\operatorname{div}_{g_0}(x-\Phi^{-1}(y)).
\end{align*}
By Proposition \ref{prop-greenfunctionexpansion}, $G(x,y)(x-y)$ and $D_x G(x,y)(x-y)$ are bounded in terms of $g_0$.  So by \eqref{eq-modifiedgreenclosetogreen} the first two terms are bounded.  In order to show (ii) we need to bound the term
\begin{align*}
    |4G(\Phi(x),y) - \tilde G(x,y)\psi\paren{\frac{d_{g_0}(x,\Phi^{-1}(y))}{r}}\operatorname{div}_{g_0}(x-\Phi^{-1}(y))|.
\end{align*}
On the support of $\psi$,
\begin{align*}
    \diver_{g_0}(x-\Phi^{-1}(y)) = 4 + O(d_{g_0}(x,\Phi^{-1}(y)))
\end{align*}
and so by \eqref{eq-modifiedgreenclosetogreen} and Proposition \ref{prop-greenfunctionexpansion} we can conclude that
\begin{align*}
    |4G(\Phi(x),y) - \tilde G(x,y)\psi\paren{\frac{d_{g_0}(x,\Phi^{-1}(y))}{r}}\operatorname{div}_{g_0}(x-\Phi^{-1}(y))| \le C.
\end{align*}

To show (iii) note that
\begin{align*}
    \|SV\|_\infty  \le &\|D_x \tilde G(x,y) \psi\paren{\frac{d_{g_0}(x,\Phi^{-1}(y))}{r}}(x-\Phi^{-1}(y))\|_\infty \\
    &+ \|\tilde G(x,y)S(\psi\paren{\frac{d_{g_0}(x,\Phi^{-1}(y))}{r}} (x-\Phi^{-1}(y)))\|_\infty.
\end{align*}
The first term is bounded in terms of $C(g_0)^K$ by Equation \eqref{eq-phigreenfunctionderivbound}.
For the second term, we first show that
\[
\|S(\psi\paren{\frac{d_{g_0}(x,\Phi^{-1}(y))}{r}}(x-\Phi^{-1}(y)))\|_\infty < C(g_0) d_{g_0}(x,\Phi^{-1}(y)).
\]
This is clear if the derivative lands on $\psi$.  The only term to consider then is $S(x - \phi^{-1}(y))$, which can be written in coordinates around $\Phi^{-1}(y)$ as
\begin{align*}
    S(x-\Phi^{-1}(y))_{ij} &= \frac{1}{2}(\partial_i(x^l) (g_0)_{l j} + \partial_j(x^l) (g_0)_{i l} + (x^l-\Phi^{-1}(y)^l)(\partial_l g_0)_{ij} - \frac{1}{4} \operatorname{div}_{g_0}(x-\Phi^{-1}(y)) (g_0)_{ij} \\
    &= (g_0)_{ij} - \frac{1}{4} \operatorname{div}_{g_0}(x-\Phi^{-1}(y)) (g_0)_{ij} + (x^l-\Phi^{-1}(y)^l)(\partial_l g_0)_{ij} \\
    &= -(x-\Phi^{-1}(y))^l\frac{\partial \log(\sqrt{\det g_0})}{\partial x^l}(g_0)_{ij} + (x^l-\Phi^{-1}(y)^l)(\partial_l g_0)_{ij}.
\end{align*}
So
\begin{align*}
    |S(x-\Phi^{-1}(y))_{ij}| \le C(g_0) d_{g_0}(x,\Phi^{-1}(y)).
\end{align*}
And therefore,
\begin{align*}
    |\tilde G(x,y)S(\psi\paren{\frac{d_{g_0}(x,\Phi^{-1}(y))}{r}} (x-\Phi^{-1}(y)))| &\le  \paren{\frac{1}{8\pi^2} \log \frac{1}{\rho(x,\Phi^{-1}(y))} + h(x,y)}C(g_0) d_{g_0}(x,\Phi^{-1}(y))\\
   &\le C(g_0)^K d_{g_0}(x,\Phi^{-1}(y))\log \frac{1}{d_{g_0}(x,\Phi^{-1}(y))} \\
   & \le C(g_0)^K,
\end{align*}
where the second inequality is due to Lemma \ref{lemma-smoothdistance}, (i).
This gives that
\begin{align*}
    \|SV\|_\infty \le C(g_0)^K.
\end{align*}
\end{proof}

We are now ready to define the vector field along which we will flow to construct the bilipschitz map. 
In what follows define for any $\eta \in C^\infty(S^4)$
\begin{align*}
    L\eta(x) := \int_{S^4} G(x,y) \eta(y) dv_0(y).
\end{align*}
\begin{lemma}\label{lemma-finalvectorfield}
Let $\eta \in C^\infty(S^4)$.  There exists a continuous vector field $v \colon S^4 \to S^4$ with integrable derivatives such that
\begin{align*}
    (L\eta)\circ \Phi = \frac{1}{4} \operatorname{div} v + b.
\end{align*}
Additionally,
\begin{itemize}
    \item[(i)] $\|b\|_\infty < C(g_0,K) \|\eta\|_1$ and
    \item[(ii)] $\|Sv\|_\infty <  C_1(g_0)^K \|\eta\|_1$,
    where $C_1(g_0)$ does not depend on $K$.
\end{itemize}
\end{lemma}
\begin{proof}
Define
\begin{align*}
     v(x) = \int_{S^4} V(x,y) \eta(y) dv_0(y),
\end{align*}
where $V(x,y)$ was constructed in Lemma \ref{lemma-potentialvectorfield}.  The definition and Lemma \ref{lemma-potentialvectorfield} immediately give that $v$ is continuous with integrable derivatives and that (ii) is satisfied.

Lemma \ref{lemma-potentialvectorfield} and
\begin{align*}
    \operatorname{div} v(x) = \int_{S^4} \operatorname{div} V(x,y) \eta(y) dv_0(y)
\end{align*}
together imply (i).
\end{proof}
We will need the following technical lemma that is also used in \cite[Lemma 6.1]{bonkheinonensaksman}.  
\begin{lemma}\label{lemma-constantiterbound}
Let $\Lambda(s) = C^s$, for some constant $C \ge 1$.
For $k \in \bN$, suppose that there exists positive constants $M_k(j)$ for $j \in \{0,\dots,k\}$ such that $M_k(0) = 0$ and
\begin{align*}
   0 \le M_k(j)-M_k(j-1) \le \frac{\epsilon}{k}\Lambda(M_k(j-1)),
\end{align*}
where $\epsilon < \int_0^\infty \frac{ds}{\Lambda(s)}$.
Then $M_k(k) \le M_0(1)$, where $M_0(1)$ is the unique solution to
\begin{align*}
    \frac{d}{ds}M_0(s) = \epsilon \Lambda(M_0(s)),\quad M_0(0) = 0.
\end{align*}
\end{lemma}
\begin{proof}
First note that $M_0(s)$ exists for $0 \le s \le 1$ since $\epsilon < \int_0^\infty \frac{ds}{\Lambda(s)}$.
We prove by induction on $j$ that $M_k(j) \le M_0(j/k)$.  When $j = 0$ this is clear.  We have that
\begin{align*}
    M_k(j) &\le M_k(j-1) + \frac{\epsilon}{k}\Lambda(M_k(j-1)) \\
    &\le M_0((j-1)/k) + \frac{\epsilon}{k}\Lambda(M_0((j-1)/k)).
\end{align*}
by induction and the monotonicity of $\Lambda$.  Applying the monotonicity of $\Lambda$ again we have that,
\begin{align*}
    M_k(j) &\le M_0((j-1)/k) + \epsilon\int_{(j-1)/k}^{j/k} \Lambda(M_0(s))ds \\
    &= M_0(j/k).
\end{align*}
\end{proof}
We note that the construction of $M_0$ gives that
\begin{align}\label{eq-boundoforM}
    M_0(1) \le C' \epsilon,
\end{align}
where $C'$ depends on the constant $C$ in Lemma \ref{lemma-constantiterbound}.

\begin{proof}[Proof of Theorem \ref{thm-qcflowbilip}]

In order to prove Theorem \ref{thm-qcflowbilip}
we would like to solve
\begin{align}\label{eq-ODE}
    \frac{d}{dt} f_t(x) = v_t(f_t(x)),
\end{align}
where
\begin{align*}
    \frac{1}{4}\operatorname{div}v_t \circ f_t  = L\eta + \text{bounded term},
\end{align*}
for a given $\eta \in C^\infty(S^4)$.
Where we recall that
\begin{align*}
    L\eta = \int_{S^4} G(x,y)\eta(y)dv_0(y).
\end{align*}
The divergence condition for $v$ depends on $f_t$ so we need to define the vector field simultaneously with the solution to the ODE.
Following the proof in \cite{bonkheinonensaksman}, we discretize the time interval.
Let $v$ be the vector field from Lemma \ref{lemma-finalvectorfield}, with $\Phi(x) = x$.  By Proposition \ref{prop-qcflowmfd}, there exists a solution to the ODE in \eqref{eq-ODE} for time $[0,1/k]$, where $k \in \bb N$. 
The solutions are quasiconformal, which we denote $f_{1/k,k}$. We next construct a new vector field $v_{1/k,k}$, which is again defined from Lemma \ref{lemma-finalvectorfield}, with the quasiconformal map $\Phi = f_{1/k,k}^{-1}$.

We continue this way until we get a map $f_{1,k}$.
That is, for each time period $[\frac{j}{k},\frac{j+1}{k}]$ we solve the ODE
\begin{align*}
    \frac{d}{dt}\Psi_{t,j,k}(x) = v_{j/k,k}(\Psi_{t,j,k}(x))
\end{align*}
with $\Psi_{0,j,k}(x) = x$. We then define $f_{(j+1)/k,k} := \Psi_{1/k,j,k} \circ f_{j/k,k}$.
This can be done for all $k \in \mathbb N$ and so we have a family of quasiconformal maps $f_{1,k}$, for $k \in \bb N$. 
We would like to show that they have a bounded quasiconformal constant and that there is a converging subsequence.  We will then show that the limiting map will satisfy that $\log J_f = L\eta + \text{a bounded term}$.

We first show that the quasiconformal constant stays bounded along each iteration.  Define the increasing function $C \colon [1,\infty) \to [1,\infty) $ as the function that takes a quasiconformal dilatation $H$ and outputs $C_1(g_0)$ from Lemma \ref{lemma-finalvectorfield}, (ii).  The function $C$ depends only on $g_0$.
By Lemma \ref{lemma-finalvectorfield}, 
\begin{align*}
    \|Sv_{1/k,k}\|_\infty < C(g_0)^{H\paren{1/k}}\|\eta\|_1,
\end{align*}
where $v_{1/k,k}$ is defined as above and $H(1/k)$ is the quasiconformal constant of $f_{1/k,k}$.
If we flow along the ODE in \eqref{eq-ODE} to construct $f_{2/k,k}$, then $f_{2/k,k}$ will be quasiconformal with a constant
\begin{align*}
    e^{\frac{4}{k} C(g_0)^{H(1/k)}\|\eta\|_1}H\paren{1/k},
\end{align*}
by Proposition \ref{prop-qcflowmfd}.
Inductively,
\begin{align*}
    \|Sv_{j/k,k}\|_\infty < C(g_0)^{H(j/k)}\|\eta\|_1
\end{align*}
and the quasiconformal constant of $f_{(j+1)/k,k}$ satisfies
\begin{align*}
    H((j+1)/k) \le H(j/k) e^{\frac{4}{k} C(g_0)^{H(j/k)}\|\eta\|_1}.
\end{align*}
So
\begin{align*}
    \log H((j+1)/k) - \log H(j/k) \le \frac{4}{k}C(g_0)^{H(j/k)}\|\eta\|_1.
\end{align*}
So if $4\|\eta\|_1 < \int_0^\infty C(g_0)^{-s}ds$, then by Lemma \ref{lemma-constantiterbound}, $H(j/k)$ is bounded independently of $j$ and $k$.  In fact, we have that
\begin{align*}
    H(1) \le C'\|\eta\|_1,
\end{align*}
by \eqref{eq-boundoforM}, where $C'$ depends only on $g_0$.

By Proposition \ref{prop-qcmapnormalfamily}, there exists a uniformly convergent subsequence of $\{f_{1,k}\}_{k \in \bb N}$.  
The Jacobians of $f_{1,k}$ will converge weakly as well (see \cite[p. 159]{rickmanbook}, the proof is for the standard metric, but applies in our setting as well).
To make this more precise, consider the behavior of $J_{j/k,k} := J_{f_{j/k,k}}$.
By Proposition \ref{prop-jacobianflowmfd}, the Jacobian of $f_{1,k}$ satisfies
\begin{align}\label{eq-decomposeofJ1}
   \log  J_{1,k} = \sum_{j=0}^{k-1} \int_{j/k}^{(j+1)/k} \operatorname{div}(v_{j/k,k})\circ f_{t} dt
\end{align}
and, by Lemma \ref{lemma-finalvectorfield},
\begin{align*}
    \operatorname{div}(v_{j/k,k})  = 4(L\eta) \circ f_{j/k,k}^{-1} + b_{j,k},
\end{align*}
where $b$ is bounded in terms of $g_0$ and $H = \sup_{j,k} H(j/k)$.  
Let $t \in [j/k,(j+1)/k]$,
\begin{align*}
    |4(L\eta)\circ f_{j/k,k}^{-1}(f_t(x)) - 4L\eta(x)| \le C(\eta,g,H) |f_{j/k,k}^{-1}(f_t(x)) - x|,
\end{align*}
since $\eta$ is smooth and so $L\eta$ is also smooth and therefore Lipschitz on $S^4$.
\begin{align}\label{eq-calcbeforelemma}
    |f_{j/k,k}^{-1}(f_t(x)) - x| = |f_{j/k,k}^{-1}(y) - f_t^{-1}(y)| \le C(\eta,g,H) \|v_{j/k,k}\|_\infty \frac{1}{k}.
\end{align}
Recall that
\begin{align*}
    v_{j,k}(x) = \int V_{j,k}(x,y)\eta(y)dy = \int  \paren{\frac{1}{8\pi^2}\log \frac{1}{\rho_{j,k}(x,f_{j/k,k}^{-1}(y))} + h(x,y)}\eta(y)dy.
\end{align*}
We will bound $v_{j,k}$ uniformly.
The term with $h$ is independent of $f_{j,k}$ and is therefore bounded in terms of $\eta$.    To bound the other term, we need the following lemma.
\begin{lemma}\label{lemma-uniformlogboundeta}
Let $\eta \in C^\infty(S^4)$ and let $r > 0$ be less than the injectivity radius of $g_0$. For any $x \in S^4$,
\begin{align*}
   \bigg | \int_{B(x,r)} \log \bigg(\frac{1}{d_{g_0}(x,y)}\bigg )\eta(y)dy \bigg | < C(\eta,g_0),
\end{align*}
where the constant is independent of $x$.
\end{lemma}
\begin{proof}
The metric $g_0$ satisfies that on $B(x,r)$, $d_{g_0}(x,y)$ is comparable to $|x-y|_{\text{Eucl}}$, where the constant of comparability depends on $g_0$ and not $x$.  So
\begin{align*}
    \bigg | \int_{B(x,r)} \log \bigg(\frac{1}{d_{g_0}(x,y) }\bigg )\eta(y)dy \bigg | \le C(g_0) \|\eta\|_\infty \int_{B(x,r)} \log \bigg(\frac{1}{|x-y|_\text{Eucl}}\bigg ) < C(\eta,g_0).
\end{align*}
\end{proof}

We now proceed to bound $v_{j,k}(x)$.  Recall that
\begin{align*}
     v_{j,k}(x) = \int \paren{\frac{1}{8\pi^2}\log \frac{1}{\rho_{j,k}(x,f_{j/k,k}^{-1}(y))} + h(x,y)}\eta(y)dy.
\end{align*}
We need only bound
\begin{align*}
     A_{j,k}(x) = \int_{B(x,r)} \log \frac{1}{\rho_{j,k}(x,f_{j/k,k}^{-1}(y))}\eta(y) dy.
\end{align*}
We do this inductively in $j$.  When $j = 0$, by Lemma \ref{lemma-smoothdistance}
\begin{align*}
    |A(x)| \le \int_{B(x,r)} \log \frac{1}{d_{g_0}(x,y)}\eta(y)dy,
\end{align*}
where $r$ is the injectivity radius for $g_0$.  This is bounded in terms of $\eta$ and $g_0$ by Lemma \ref{lemma-uniformlogboundeta}.  So for $x \in B(0,1)$, we have that $|v_{0,k}(x)| \le M_0$.
Assume that $|v_{j,k}| < M$ and $|f_{j,k}(0)| < \frac{j}{k}M$, where $M$ will be chosen below.  By Lemma \ref{lemma-smoothdistance},
\begin{align*}
    A_{j+1,k}(x) &= \int \log \frac{1}{\rho_{j,k}(x,f_{(j+1)/k,k}^{-1}(y))}\eta(y) dy \\
    &\sim \int \log \frac{1}{d_{g_0}(f_{(j+1)/k,k}(x), y)}\eta(y) dy,
\end{align*}
where the constant only depends on the quasiconformal constant $H(j/k)$ of $f_{j/k,k}$ and $g_0$.  We also know that $H(j/k)$ is bounded independently of $j$ and $k$.  So by Lemma \ref{lemma-uniformlogboundeta},
\begin{align*}
    |A_{j+1,k}(x)| \le C(\eta,g_0,H)
\end{align*}
and therefore $v_{j+1,k}$ is bounded by a constant that depends only on $\eta$, $g_0$ and $H$.

Returning to our calculation in Equation \eqref{eq-calcbeforelemma}, we have that
\begin{align*}
    |4(L\eta)\circ f_{j/k,k}^{-1}(f_t(x)) - 4L\eta(x)|  \le C \frac{1}{k}.
\end{align*}
By \eqref{eq-decomposeofJ1} and Lemma \ref{lemma-finalvectorfield},
\begin{align*}
    |\log  J_{1,k} - 4(L\eta)| \le \sum_j \int_{j/k}^{(j+1)/k} |b_{j,k} | dt + \frac{C}{k} \le C(g_0)\|\eta\|_1 + \frac{C}{k}.
\end{align*}
Note that Lemma \ref{lemma-finalvectorfield} gives that the constant $C(g_0)$ should depend on $H(j/k)$. However, $H(j/k) \le M_0(1)$, where $M_0$ is the function given in Lemma \ref{lemma-constantiterbound}, and $M_0(1)$ is bounded in terms of $g_0$.
Continuing our calculation, we get that
\begin{align*}
    e^{-C(g_0)\|\eta\|_1}e^{-C/k} e^{4L\eta} \le J_{1,k} \le e^{C(g_0)\|\eta\|_1} e^{C/k}e^{4L\eta}.
\end{align*}
The uniform limit of quasiconformal maps implies the weak convergence of the Jacobians of those maps (see again \cite[p. 159]{rickmanbook}). 

So for any smooth function $\psi$ on $S^4$, we have that
\begin{align*}
    \int_{S^4} \psi J_{1,k}dv_0 \to \int_{S^4} \psi J_{f}dv_0.
\end{align*}
The function $J_{1,k}$ is positive and so
\begin{align*}
     e^{-C(g_0)\|\eta\|_1}\int_{S^4} \psi e^{4L\eta}dv_0 \le \int_{S^4} \psi J_fdv_0 \le e^{C(g_0)\|\eta\|_1}\int_{S^4} \psi e^{4L\eta}dv_0.
\end{align*}
This gives that for almost every $x \in S^4$,
\begin{align}
    e^{-C(g_0)\|\eta\|_1}e^{4L\eta(x)} \le J_f(x) \le   e^{C(g_0)\|\eta\|_1}e^{4L\eta(x)}.
\end{align}
Besides the comparability of the Jacobian to $e^{4L\eta}$, this additionally shows that the limiting map $f$ is non-constant.
To finish the proof, it suffices to show that $L\eta = w$, when $\eta = Q_1 e^{4w} - Q_0$.  Recall that
\begin{align*}
    L\eta(x) = \int_{S^4} G(x,y)\eta(y)dv_0(y),
\end{align*}
where $G$ is the Green's function for the Paneitz operator.  Additionally,
\begin{align*}
    P w + Q_0 = Q_1 e^{4w}.
\end{align*}
So $L\eta = w + c$, where $c$ is some constant.

So $f$ is quasiconformal and the Jacobian of $f$ is comparable to $e^ce^{4w}$ with a constant of comparability $e^{C(g_0)\alpha}$.
We still need to show that $e^c \sim 1$.  If we normalize the volume of $(S^4,g_1)$ and $(S^4,g_0)$ to be $1$, we have that
\begin{align*}
   e^c &= e^c \int_{S^4}e^{4w}dv_0 \\
   &\sim \int_{S^4} f^*(dv_0) \\
   &=\int_{S^4}dv_0 = 1.
\end{align*}

Proposition \ref{prop-qcandcomparablejacobianbilip} yields that $f$ is additionally bilipschitz with the constant $C(g_0)e^{C(g_0)\alpha}$.  Since we assume that $\alpha$ is small we see that in fact the bilipschitz constant is less than $1 + C\alpha$, where $C$ depends only on $g_0$.
\end{proof}

\section{Preliminaries for Theorem \ref{thm-riccibilip}}\label{sec-ricciproofprelim}

The strategy for proving Theorem \ref{thm-riccibilip} involves controlling the $L^p$-norm of $W$ and $E$ for $p > 2$, where $W$ is the Weyl tensor and $E$ is the traceless Ricci tensor.
In order to do this we consider the Bach tensor, which is defined in coordinates as
\begin{align}\label{bach}
    B_{ij} = \nabla^k\nabla^l W_{kijl} + \frac{1}{2}\Ric^{kl}W_{kijl}.
\end{align}
The Bach tensor is a conformally invariant tensor in that if $g = e^{2w}g_0$, then
\begin{align*}
    B_g = e^{-2w}B_{g_0}.
\end{align*}
This implies that
\begin{align*}
    \int_{S^4}|B|dv
\end{align*}
is a conformal invariant. For a discussion of the Bach tensor see e.g., \cite{derdzinski}.

We now record a conformally invariant Sobolev inequality \eqref{eq-yamabesobolev}, which follows from the definition of
the Yamabe constant.
Let $g$ be a metric on $S^4$ with a positive Yamabe constant $Y([g])$, then for all $u \in W^{1,2}(S^4)$
\begin{align}\label{eq-yamabesobolev}
    Y([g])\|u\|_4^2 \le  \|\nabla u\|^2 + \frac{1}{6}\int_{S^4} R_gu^2 dV_g.
\end{align}

The proof for the following inequality can be found in \cite{changchen2021}.
\begin{prop}[\cite{changchen2021}]\label{prop-yamabelower}
Let $g_0$ be the Yamabe metric on $S^4$ with a positive Yamabe constant, then
\begin{align}\label{eq-WcontrolsE}
    \frac{1}{2}\int_{S^4}|E|^2 dv_0\le \int_{S^4} |W|^2dv_0
\end{align}
and
\begin{align}\label{eq-yamabeconstantbound}
    24 \paren{16 \pi^2 - \int_{S^4} |W|^2dv_0} \le Y([g])^2.
\end{align}
\end{prop}

\begin{prop}\label{prop-bachcontrol}
Let $g_0$ be a Yamabe metric on $S^4$ with positive Yamabe constant.
There exists an $\epsilon >0$ so that if
\begin{align*}
    M = \int_{S^4} (|W|^2 + |E|^2) dv_0 < \epsilon^2,
\end{align*}
then
\begin{align}\label{eq-tracelessriccil4}
    \int_{S^4} |E|^4 dv_0\le C M \int_{S^4}|B|^2dv_0
\end{align}
and
\begin{align}\label{eq-weyll4}
    \int_{S^4} |W|^4 dv_0\le C M \int_{S^4}|B|^2dv_0,
\end{align}
where $C > 0$ does not depend on $g_0$.
\end{prop}
\begin{proof}

The proof of the two inequalities \eqref{eq-tracelessriccil4} and \eqref{eq-weyll4} depends
on two equivalent expressions of the Bach tensor via the second Bianchi identity which we now explain.

Start with Bach tensor as defined in \eqref{bach}.  
A well-known relationship between
the Weyl tensor and the Ricci tensor via the Bianchi identity, a proof of which can also be found in \cite[p. 124]{changgurskyyang2003}, is that
\begin{align}\label{weyldA}
  \nabla^k W_{ijlk} := (\delta W)_{ijl} = \frac{1}{2}(dA)_{ijl}.
\end{align}
where
\begin{align*}
    (dA)_{ijl} = \nabla_{i}A_{jl} - \nabla_j A_{il},
\end{align*}
and $A_{ij} = \operatorname{Ric}_{ij}-\frac{1}{6}Rg_{ij}$.  
Applying \eqref{weyldA} to the expression of the
Bach tensor \eqref{bach}, we find that (\cite[p. 717]{changgurskyyang2002})
\begin{align} \label{bach-1}
    B_{ij} &= -\frac{1}{2}\Delta E_{ij} + \frac{1}{6} \nabla_i\nabla_j R - \frac{1}{24}\Delta R g_{ij} - E^{kl}W_{ikjl} +E_{i}^kE_{jk}- \frac{1}{4}|E|^2 g_{ij} + \frac{1}{6}R E_{ij} \\
    &= -\frac{1}{2}\Delta E_{ij}  - E^{kl}W_{ikjl} +E_{i}^kE_{jk}- \frac{1}{4}|E|^2 g_{ij} + \frac{1}{6}R E_{ij},
\end{align}
since $R$ is constant.
If we multiply \eqref{bach-1}  by $E$ and integrate, then
\begin{align*}
    \int |\nabla E|^2 &= 2 \int EB + 2\int EEW- 2\int \operatorname{Tr} E^3 + \frac{1}{2}\int |E|^2 \operatorname{Tr}E - \frac{1}{3} \int R |E|^2 \\
    &=2 \int EB + 2\int EEW- 2\int \operatorname{Tr} E^3  - \frac{1}{3} \int R |E|^2,
\end{align*}
since $\operatorname{Tr}(E) = 0$.  By \eqref{eq-yamabesobolev},
\begin{align*}
    Y([g])\bigg ( \int |E|^4 \bigg )^{1/2}& \le \int |\nabla|E||^2  + \frac{1}{6} \int R|E|^2 \\
    &\le \int |\nabla E|^2  + \frac{1}{6} \int R|E|^2.
\end{align*}
So
\begin{align*}
    Y([g])\bigg ( \int |E|^4 \bigg )^{1/2} &\le 2 \int EB + 2\int EEW- 2\int \operatorname{Tr} E^3  - \frac{1}{6} \int R |E|^2 \\
    &\le 2\|B\|_2\|E\|_2 + 2\|W\|_2\|E\|_4^2 + 2\|E\|_4^2\|E\|_2,
\end{align*}
by the Cauchy-Schwarz inequality and since $R > 0$.  By our assumption on $\|E\|_2$ and $\|W\|_2$,
\begin{align*}
    Y([g])\|E\|_4^2 \le 2M^{1/2}\|B\|_2 + 4 M^{1/2} \|E\|_4^2.
\end{align*}
So if $\epsilon < Y([g])/4$, then
\begin{align*}
    \|E\|_4^2 \le \frac{2 M^{1/2}}{Y([g]) - 4M^{1/2}} \|B\|_2.
\end{align*}
The constant is bounded by $CM^{1/2}$, where $C$ is independent of $M$ and $Y([g])$, by Equations \eqref{eq-WcontrolsE} and \eqref{eq-yamabeconstantbound}.
Hence the inequality in \eqref{eq-tracelessriccil4} is satisfied.

We next prove the inequality in \eqref{eq-weyll4}.  
Starting from the definition of Bach tensor \eqref{bach}, 
\begin{align*}
    \int B_{ij}A_{ij} &= \int A_{ij}(\nabla_k\nabla_l W_{iklj} - \frac{1}{2}A_{kl}W_{ikjl}) \\
    &= \int -\nabla_k A_{ij}\nabla_l W_{iklj} - \frac{1}{2}A_{ij}A_{kl}W_{ikjl} \\
    &=\int -\frac{1}{2}(\nabla_k A_{ij} - \nabla_iA_{kj})\nabla^l W_{iklj} - \frac{1}{2}A_{ij}A_{kl}W_{ikjl},
\end{align*}
by the symmetry of $W$. Therefore, from 
the \eqref{weyldA},
\begin{align*}
    \int B_{ij}A_{ij}  = \int |\delta W|^2 -\frac{1}{2}\int A_{ij}A_{kl}W_{ikjl}.
\end{align*}
Also,
\begin{align*}
    \int |\nabla W|^2 =4 \int |\delta W|^2+ 72\int \det W - \frac{1}{2} \int R |W|^2,
\end{align*}
see \cite[p. 126]{changgurskyyang2003}.  So
\begin{align*}
    \int |\nabla W|^2  = 4\int BA + 2\int WAA + 72\int \det W - \frac{1}{2} \int R |W|^2.
\end{align*}
Since the Weyl tensor has zero trace we have that $WAA = WEE$ and that $BA = BE$.  Substituting this we arrive at
\begin{align}\label{eq-gradweylrep}
     \int |\nabla W|^2  = 4\int BE + 2\int WEE + 72\int \det W - \frac{1}{2} \int R |W|^2.
\end{align}
We now can prove \eqref{eq-weyll4} in a similar way as the proof for \eqref{eq-tracelessriccil4}.
By \eqref{eq-yamabesobolev} and  \eqref{eq-gradweylrep}
\begin{align*}
    Y([g]) \bigg (\int |W|^4\bigg )^{1/2} &\le \int |\nabla |W||^2 + \frac{1}{6}\int R|W|^2 \\
    &\le 4\int BE + 2\int WEE + 72\int \det W - \frac{1}{3} \int R |W|^2 \\
    &\le \int BE + 2\int WEE + 72\int \det W,
\end{align*}
since $R > 0$.  By the Cauchy-Schwarz inequality,
\begin{align*}
    Y([g]) \bigg (\int |W|^4\bigg )^{1/2} &\le C \|W\|_2(\|W\|_4^2 + \|E\|_4^{1/2}) + \|E\|_2\|B\|_2 \\
    &\le C(M^{1/2}\|W\|_4^2 + M^{1/2}(1+M^{1/2}) \|B\|_2),
\end{align*}
by \eqref{eq-tracelessriccil4}.  So if $\epsilon < Y([g])/C$, then
\begin{align*}
    \|W\|_4^{1/2} \le C\frac{M^{1/2}(1+M^{1/2})}{Y([g]) - CM^{1/2}}\|B\|_2.
\end{align*}
The coefficient of $\|B\|_2$ is controlled by a constant times $M^{1/2}$, independently of $\epsilon$ and $Y([g])$, by Equations \eqref{eq-WcontrolsE} and \eqref{eq-yamabeconstantbound}.
So, the inequality in \eqref{eq-weyll4} is satisfied.
\end{proof}

\section{Proof of Theorem \ref{thm-riccibilip}}\label{sec-proofricci}
With these preliminaries, we now proceed with the proof for Theorem \ref{thm-riccibilip}. For the reader's convenience we restate the theorem.
\begin{named}{Theorem \ref{thm-riccibilip}}
Let $g_0$ be a Yamabe metric of constant and positive scalar curvature on $S^4$, let
\begin{align*}
    \beta = \int_{S^4}|W|^2 dv_0,
\end{align*}
and let
\begin{align*}
   \gamma = \int_{S^4}|B|^2 dv_0,
\end{align*}
where $W$ is the Weyl tensor and $B$ is the Bach tensor.
There exists a constant $\epsilon_2 > 0$ such that if $\beta < \epsilon_2$, then 
$(S^4,g_0)$ is bilipschitz equivalent to $(S^4,g_c)$, where $g_c$ is the standard metric for $S^4$. That is, for all $x,y \in \bR^4$ there exists $L_2 > 0$ such that
\begin{align*}
  \frac{1}{L_2}d_{g_c}(x,y) \le  d_{g_0}(x,y) \le L_2d_{g_c}(x,y),
\end{align*}
where $L_2 \le 1 + C'\gamma\beta$.
\end{named}
In order to prove Theorem \ref{thm-riccibilip}, we will apply the normalized Ricci flow to $(S^4,g_0)$ and show that under the flow the metric converges sufficiently rapidly to $g_c$.
Results of this type have been shown previously and the proofs will not be provided.
The normalized Ricci flow is defined by
\begin{align}\label{eq-normalizedricciflow}
    \frac{d g}{dt} = -2 \operatorname{Ric_g} + \frac{1}{2}\overline{R_g}g,
\end{align}
where
\begin{align*}
    \overline{R} = \frac{1}{\operatorname{vol}_g(S^4)}\int_{S^4} R_g dv_g
\end{align*}
and $R_g$ is the scalar curvature of $g$. 

We now recall a result of Gursky \cite{gursky94}.
\begin{theorem}\label{thm-gursky} 
On $S^4$, if $g_0$ has a positive Yamabe constant and
\begin{align*}
    \int_{S64} |W|^2 < \epsilon_2,
\end{align*}
where $W$ is the Weyl tensor for $g_0$, then
the solution to \eqref{eq-normalizedricciflow} with initial condition $g(0) = g_0$ exists for all time and converges to $g_c$.
\end{theorem}

We remark that in the proof of Gursky, he first established the convergence of the metric for a short time to a metric with bounded curvature that satisfies a pinching condition.
He then quoted an earlier work of Huisken \cite{huisken} to establish the exponential convergence of the metric to $g_c$ under the flow (see also \cite{margerin}). 

In a recent work of Chang and Chen \cite{changchen2021}, they have established the monotonicity of the $L^p$-norm of a certain curvature quantity under the flow.  
They use this to give a quantitative estimate of the result of the Gursky. More precisely they have established the following result.

\begin{theorem} \label{thm-chang-chen} If $(S^4, g(t))$ is a normalized Ricci 
flow starting from a unit volume
positive Yamabe metric $(S^4,  g_0)$ with 
\begin{align*}
\int_{S^4} |W|_{g_0}^2
 dv_{g_{0}} < 10^{-3} \pi^2
\end{align*}
then for all $t > 0$ 
and $ p \in [2, 2 + \frac{1}{3}],$
\begin{align*}
\||W|+ |E|+|R - \bar  R|\|_{\infty} (t) \leq C_1( 1+ t^{- \frac{2}{p}}) e^{-C_2 t} 
\||W|+|E|)\|_{p} (0),
\end{align*}
where $C_1, C_2$ are dimensional constants and independent of $g_0$.
\end{theorem}

\begin{proof}[Proof of Theorem \ref{thm-riccibilip}]
We first show that $\|G(0)\|_4$ is controlled by $\beta$.  Recall that $G(0) = |W| + |E| + |R - \overline{R}|$.  At time $0$ the scalar curvature is constant and so we have that $G(0) = |W| + |E|$.
By Proposition \ref{prop-bachcontrol}, $\|G(0)\|_4 \le C \|G(0)\|_2 \|B(0)\|_2$, where $B$ is the Bach tensor at time $0$. 
So it suffices to show that $\|G(0)\|_2$ is controlled by $\beta$.
The Chern-Gauss-Bonnet formula implies that
\begin{align*}
    16\pi^2 = \int_{S^4} (|W|^2 - \frac{1}{2} |E|^2 +\frac{1}{24}R^2 ) dv_{g_0}.
\end{align*}
Since $R$ is constant, we have that 
\begin{align*}
    \int_{S^4} R^2dv_{g_0} = Y([g_0])^2.
\end{align*}
In \cite{aubin76}, Aubin showed that the Yamabe constant is largest for $g_c$ and so
\begin{align*}
    \frac{1}{24}\int_{S^4} R^2dv_{g_0} \le 16 \pi^2.
\end{align*}
This, together with the Chern-Gauss-Bonnet formula gives that
\begin{align*}
    \frac{1}{2}\int_{S^4} |E|^2dv_{g_0} \le \beta
\end{align*}
and so
\begin{align*}
    \|G(0)\|_2^2 \le 3\beta.
\end{align*}
The normalized Ricci flow equation gives that
\begin{align*}
    \abs{\frac{dg}{dt}} &\le 2|\operatorname{Ric} - \frac{1}{4}Rg| + \frac{1}{2}|(R-\overline{R})g| \\
    &\le 2 \|G(t)\|_\infty.
\end{align*}
By Theorem \ref{thm-chang-chen}, 
\begin{align*}
    \abs{\frac{dg}{dt}} \le C t^{-2/(2+\delta)}e^{-C't}\|G(0)\|_{2 + \delta}.
\end{align*}
By H\"older's inequality and Proposition \ref{prop-bachcontrol},
\begin{align*}
    C t^{-2/(2+\delta)}e^{-C't}\|G(0)\|_{2 + \delta} &\le C t^{-2/(2+\delta)}e^{-C't}\|G(0)\|_{4} \\
    &\le C t^{-2/(2+\delta)}e^{-C't}\|B(0)\|_2\|G(0)\|_{2}
    \\
    & \le C t^{-2/(2+\delta)}e^{-C't}\|B(0)\|_2\beta.
\end{align*}
where $C$ is a positive constant that may differ in each line.
Since $2/(2+ \delta) < 1$, by integrating $dg /  dt$,
\begin{align*}
    \|g(t)-g_0\|_\infty \le C\gamma \beta,
\end{align*}
where the constant does not depend on $g_0$.

So $g_0$ is bilipschitz equivalent to $g(t)$ for all $t > 0$.  By Theorem \ref{thm-gursky},  $\lim_{t\to \infty} g(t) = g_c$ and so $g_0$ is bilipschitz equivalent to $g_c$.
\end{proof}

\bibliographystyle{siam}
\bibliography{qcurvbl}

\end{document}